\documentclass[12pt]{article}
\usepackage{amsfonts}
\usepackage{amsmath}
\usepackage{geometry}
\usepackage{titlesec}
\usepackage{lineno}

\newtheorem{theorem}{Theorem}

\newenvironment{proof}[1][Proof]{\noindent\textbf{#1.} }{\ \rule{0.5em}{0.5em}}
\input{tcilatex}
\begin{document}

\title{\textbf{A Single Server Retrial Queue with Different Types of Server
Interruptions}}
\author{T\textsc{ewfik} K\textsc{ernane} \\
\textit{Department of Mathematics}\\
\textit{Faculty of Sciences}\\
\textit{King Khalid University}\\
\textit{Abha, Saudi Arabia}\\
tkernane@gmail.com}
\maketitle

\begin{abstract}
We consider a single server retrial queue with the server subject to
interruptions and classical retrial policy for the access from the orbit to
the server. We analyze the equilibrium distribution of the system and obtain
the generating functions of the limiting distribution.
\end{abstract}

\section{Introduction}

Queueing systems with retrials of the attempts are characterized by the fact
that an arrival customer who finds the server occupied is obliged to join a
group of blocked customers, called orbit, and reapply after random intervals
of time to obtain the service. These systems are useful in the stochastic
modeling of much of situations in practice. We can find them in aviation,
where a plane which finds the runway occupied remakes its attempt of landing
later and we say in this case that it is in orbit. In Telephone systems
where a telephone subscriber who obtains a busy signal repeats the call
until the required connection is made. In data processing, we find them in
protocol of access CSMA/CD. They appear in the modeling of the systems of
maintenance and the problems of repair among others. For details on these
models, see the book of Falin and Templeton \cite{Fal} or the recent book by
Artalejo and Gomez-Corral \cite{Art2}.

We study in this article single server retrial queues with various types of
interruptions of the server. From a practical point of view, it is more
realistic to consider queues with repetitions of calls and the server
exposed to random interruptions. Queueing models with interruptions of
service proved to be a useful abstraction in the situations where a server
is shared by multiple queues, or when the server is subject to breakdowns.
Such systems were studied in the literature by many authors. Fiems et 
\textit{al}. \cite{Fie} considered an M/G/1 queue with various types of
interruptions of the server and our work is a generalization to the case of
retrial queues. White and Christie \cite{Whi} were the first to study queues
with interruptions of service by considering a queueing system with
exponentially distributed interruptions. Times of interruptions and services
generally distributed are considered by Avi-Itzhak and Naor \cite{Avi} and
Thiruvengadam \cite{Thi}. Other generalisations were considered in the
literature ( phase-type \cite{Fed}, approximate analysis \cite{Van},
Markov-modulated environment \cite{Tak} and \cite{Mas}, processor sharing 
\cite{Nun}). Gaver \cite{Gav} considers the case where the service is
repeated or repeated and begin again after the interruption.

We consider in this paper a single server retrial queue with server
interruptions and the classical retrial policy where each customer in orbit
conduct its own attempts to get served independently of other customers
present in the orbit. We can then assume that the probability of a retrial
during the time interval $(t,t+dt)$, given that $j$ customers were in orbit
at time $t,$ is $j\theta dt+\circ \left( dt\right) $. Kulkarni and Choi \cite%
{Kul1} studied a single server linear retrial queue with server subject to
breakdowns and repairs and they obtained the generating functions of the
limiting distribution and performance characteristics. Artalejo \cite{Art1}
obtained sufficient conditions for ergodicity of multiserver retrial queues
with breakdowns and a recursive algorithm to compute the steady-state
probabilities for the M/G/1 linear retrial queue with breakdowns. The
detailed analysis for reliability of retrial queues with linear retrial
policy was given by Wang, Cao and Li \cite{Wan1}.

The remainder of paper is organized as follows. In the following section, we
describe the model and give the necessary and sufficient conditions so that
the system is stable. In section 3, we analyze the equilibrium distribution
of the system in study.

\section{Model Description}

Consider a single server queueing system in which customers arrive in
accordance with a Poisson process with arrival rate $\lambda $. If at the
instant of arrival the customer finds the server free, it takes its service
and leaves the system. Otherwise, if the server is busy or in interruption,
the arriving customer joins an unlimited queue called orbit and makes
retrials for getting served after random time intervals. We consider the
classical policy where each customer in orbit conducts his own attempts to
obtain service independently from the other customers present in the orbit.
We can then assume that the probability of a retrial during the time
interval $(t,t+dt)$, given that $j$ customers were in orbit at time $t,$ is $%
j\theta dt+\circ \left( dt\right) $. Service times constitute a series of
independent and identically distributed (i.i.d.) random variables with
common distribution function $B(t)$, density function $b(t)$, and
corresponding Laplace--Stieltjes transform (LST) $\beta (s)$ and finite
first two moments $\beta _{k}=\left( -1\right) ^{k}\beta ^{(k)}(0),$ $k=1,2$%
. Interruptions of the service may occur according to a Poisson process with
rate $\nu $ if the server is busy and this type of interruption can be
disruptive with probability $p_{d}$ (or rate $\nu _{d}=p_{d}\nu $) or
non-disruptive with probability $p_{n}=1-p_{d}$ (or rate $\nu _{n}=p_{n}\nu $%
). In the case of a disruptive interruption the customer being served
repeats his service at the end of the interruption, in the other type the
customer continues his stopped service. If the server is idle, another type
of interruptions may occur according to a Poisson process with rate $\nu
_{i} $. We call this type \textit{idle interruption}. The lengths of the
consecutive disruptive (non-disruptive, idle time) interruptions constitute
a series of i.i.d. positive random variables with distribution function $%
B_{d}(t)$ ($B_{n}(t)$, $B_{i}(t)$), density function $b_{d}(t)$ ($b_{n}(t)$, 
$b_{i}(t)$), corresponding Laplace--Stieltjes transform (LST) $\beta _{d}(s)$
($\beta _{n}(s)$, $\beta _{i}(s)$) and finite first two moments $\beta
_{k}^{d},$ ($\beta _{k}^{n},$ $\beta _{k}^{i}$) $k=1,2$.

Denote by $N(t)$ the number of customers in orbit at time $t$. Let $C(t)$ be
the state of the server at time $t$ : $C(t)=F$ if the server is free (and
functions normally), $C(t)=S$ if the server is busy (and functions
normally), $C(t)=D$ if the server is on a disruptive interruption, $C(t)=N$
if the server is on a non-disruptive interruption, $C(t)=I$ if the server is
taking an idle interruption. We introduce the random variables $\mathcal{\xi 
}(t),$ $\mathcal{\xi }_{D}(t),$ $\mathcal{\xi }_{N}(t)$ and $\mathcal{\xi }%
_{I}(t)$ defined as follows. If $C(t)=S$ then $\mathcal{\xi }(t)$ represents
the elapsed service time at time $t$; if $C(t)=D,$ then $\mathcal{\xi }%
_{D}(t)$ represents the elapsed disruptive interruption time at $t$; if $%
C(t)=N,$ then $\mathcal{\xi }_{N}(t)$ represents the elapsed non-disruptive
interruption time at $t$; and if $C(t)=I$ then $\mathcal{\xi }_{I}(t)$ is
the elapsed idle interruption time at $t.$

\section{Stability Analysis}

We first study the condition for the system to be stable. The following
theorem provides the necessary and sufficient stability condition.

\begin{theorem}
The system with classical retrial policy and interruptions is stable if and
only if the following condition is fulfilled%
\begin{equation}
\frac{\lambda \left( 1-\beta (\nu _{d})\right) }{\nu _{d}\beta (\nu _{d})}%
\left( 1+\nu _{d}\beta _{1}^{d}+\nu _{n}\beta _{1}^{n}\right) <1.
\label{stab}
\end{equation}
\end{theorem}

\begin{proof}
Let $\{s_{n};$ $n\in 
\mathbb{N}
\}$ be the sequence of epochs of service completion time. We consider the
process $Y_{n}=\left( N(s_{n}+),C(s_{n}+)\right) $ embedded immediately
after time $s_{n}.$ It is readily to see that $\{Y_{n};$ $n\in 
\mathbb{N}
\}$ is an irreducible aperiodic Markov chain. To determine the stability of
the system it remains to prove that $\{Y_{n};$ $n\in 
\mathbb{N}
\}$ is ergodic under the suitable stability condition. Let us first consider
the generalized service time $\widetilde{S}$ of a customer which includes,
in addition to the original service time $S$ of the customer, possible
interruption times during the service period of the customer.\textit{\ }%
Fiems et \textit{al}. \cite{Fie} showed that the generalized service time
has the Laplace transform%
\begin{equation*}
\widetilde{\beta }(s)=\frac{\left[ s+\nu -\nu _{n}\beta _{n}(s)\right] \beta
\left( s+\nu -\nu _{n}\beta _{n}(s)\right) }{\left[ s+\nu -\nu _{n}\beta
_{n}(s)\right] -\nu _{d}\beta _{d}(s)\left[ 1-\beta \left( s+\nu -\nu
_{n}\beta _{n}(s)\right) \right] },
\end{equation*}%
hence its expected value is given by%
\begin{equation*}
E\widetilde{S}=-\widetilde{\beta }^{\prime }(0)=\frac{\left( 1-\beta (\nu
_{d})\right) }{\nu _{d}\beta (\nu _{d})}\left( 1+\nu _{d}\beta _{1}^{d}+\nu
_{n}\beta _{1}^{n}\right) .
\end{equation*}%
For the sufficiency, we shall use Foster's criterion, which states that a
Markov chain $\{Y_{n};$ $n\in 
\mathbb{N}
\}$ is ergodic if there exists a nonnegative function $f(k),$ $k\in 
\mathbb{N}
,$ and $\delta >0$ such that for all $k\neq 0$ the mean drift%
\begin{equation}
\chi _{k}=E\left[ f(Y_{n+1})-f(Y_{n})\mid Y_{n}=k\right] ,
\end{equation}%
satisfies $\chi _{k}\leq -\delta $ and $E\left[ f(Y_{n+1})\mid Y_{n}=0\right]
<\infty .$ If we choose $f(k)=k$ we obtain 
\begin{equation*}
E\left[ f(Y_{n+1})\mid Y_{n}=0\right] =\lambda E\widetilde{S}=\frac{\lambda
\left( 1-\beta (\nu _{d})\right) }{\nu _{d}\beta (\nu _{d})}\left( 1+\nu
_{d}\beta _{1}^{d}+\nu _{n}\beta _{1}^{n}\right) <\infty ,
\end{equation*}%
and we can easily check that%
\begin{equation*}
\chi _{k}=\lambda E\widetilde{S}-1=\left[ \lambda \left( 1-\beta (\nu
_{d})\right) /\nu _{d}\beta (\nu _{d})\right] \left( 1+\nu _{d}\beta
_{1}^{d}+\nu _{n}\beta _{1}^{n}\right) -1.
\end{equation*}%
If we set%
\begin{equation*}
\delta =1-\frac{\lambda \left( 1-\beta (\nu _{d})\right) }{\nu _{d}\beta
(\nu _{d})}\left( 1+\nu _{d}\beta _{1}^{d}+\nu _{n}\beta _{1}^{n}\right)
\end{equation*}%
then the condition (\ref{stab}) is sufficient for ergodicity.\newline
To prove that the condition (\ref{stab}) is necessary, we use theorem 1 of
Sennot et \textit{al.} \cite{Sen} which states that if the Markov chain $%
\{Y_{n};$ $n\in 
\mathbb{N}
\}$ satisfies Kaplan's condition, namely $\chi _{k}<\infty $ for all $k\geq
0 $ and there is an $k_{0}$ such that $\chi _{k}\geq 0$ for $k\geq k_{0}$,
then $\{Y_{n};$ $n\in 
\mathbb{N}
\}$ is not ergodic. Indeed, if%
\begin{equation*}
\frac{\lambda \left( 1-\beta (\nu _{d})\right) }{\nu _{d}\beta (\nu _{d})}%
\left( 1+\nu _{d}\beta _{1}^{d}+\nu _{n}\beta _{1}^{n}\right) \geq 1
\end{equation*}%
then for $f(k)=k$, there is a $k_{0}$ such that $p_{ij}=0$ for $j<i-k_{0}$
and $i>0$, where $P=(p_{ij})$ is the one-step transition matrix associated
to $\{Y_{n};$ $n\in 
\mathbb{N}
\}.$\newline
The stability of the system follows from Burke's theorem (see Cooper \cite%
{Coo} p187) since the input flow is a Poisson process.
\end{proof}

\section{Steady-state analysis}

We investigate in this section the steady-state distribution of the system.
Define the functions $\mu (x),$ $\mu _{D}(x),$ $\mu _{N}(y)$ and $\mu
_{I}(x) $ as the conditional completion rates for service, disruptive
interruption, non-disruptive interruption and idle interruption,
respectively, i.e., $\mu (x)=b(x)/\left( 1-B(x)\right) ,$ $\mu
_{D}(x)=b_{d}(x)/\left( 1-B_{d}(x)\right) ,$ $\mu _{N}(x)=b_{n}(x)/\left(
1-B_{n}(x)\right) $ and $\mu _{I}(x)=b_{i}(x)/\left( 1-B_{i}(x)\right) .$

We now introduce the following set of probabilities for $j\geq 0$:%
\begin{eqnarray*}
p_{F,j}(t) &=&P\left\{ N(t)=j,\text{ }C(t)=F\right\} , \\
p_{B,j}(t,x)dx &=&P\left\{ N(t)=j,\text{ }C(t)=S,\text{ }x\leq \mathcal{\xi }%
(t)<x+dx\right\} , \\
p_{D,j}(t,x)dx &=&P\left\{ N(t)=j,\text{ }C(t)=D,\text{ }x\leq \mathcal{\xi }%
_{D}(t)<x+dx\right\} , \\
p_{N,j}(t,x,y)dy &=&P\left\{ N(t)=j,\text{ }C(t)=N,\text{ }\mathcal{\xi }%
(t)=x,\text{ }y\leq \mathcal{\xi }_{N}(t)<y+dy\right\} , \\
p_{I,j}(t,x)dx &=&P\left\{ N(t)=j,\text{ }C(t)=I,\text{ }x\leq \mathcal{\xi }%
_{I}(t)<x+dx\right\} .
\end{eqnarray*}%
where $t\geq 0$ and $x,y\geq 0.$

The usual arguments lead to the differential difference equations by letting 
$t\rightarrow +\infty $%
\begin{eqnarray}
\left( \lambda +j\theta +\nu _{i}\right) p_{F,j} &=&\int\limits_{0}^{+\infty
}\mu (x)p_{B,j}(x)dx+\int\limits_{0}^{+\infty }\mu _{I}(x)p_{I,j}(x)dx,
\label{pfj} \\
\left( \frac{\partial }{\partial x}+\lambda +\nu +\mu (x)\right) p_{B,j}(x)
&=&\left( 1-\delta _{0j}\right) \lambda
p_{B,j-1}(x)+\int\limits_{0}^{+\infty }\mu _{N}(y)p_{N,j}(x,y)dy,
\label{pbj} \\
\left( \frac{\partial }{\partial x}+\lambda +\mu _{D}(x)\right) p_{D,j}(x)
&=&\left( 1-\delta _{0j}\right) \lambda p_{D,j-1}(x),  \label{pdj} \\
\left( \frac{\partial }{\partial y}+\lambda +\mu _{N}(y)\right) p_{N,j}(x,y)
&=&\left( 1-\delta _{0j}\right) \lambda p_{N,j-1}(x,y),  \label{pnj} \\
\left( \frac{\partial }{\partial x}+\lambda +\mu _{I}(x)\right) p_{I,j}(x)
&=&\left( 1-\delta _{0j}\right) \lambda p_{I,j-1}(x).  \label{pij}
\end{eqnarray}%
With boundary conditions%
\begin{eqnarray}
p_{B,j}(0) &=&\left( j+1\right) \theta p_{F,j+1}+\lambda
p_{F,j}+\int\limits_{0}^{+\infty }\mu _{D}(x)p_{D,j}(x)dx,  \label{b0j} \\
p_{D,j}(0) &=&\nu _{d}\int\limits_{0}^{+\infty }p_{B,j}(x)dx,  \label{d0j} \\
p_{N,j}(x,0) &=&\nu _{n}p_{B,j}(x),  \label{n0j} \\
p_{I,j}(0) &=&\nu _{i}p_{F,j}.  \label{i0j}
\end{eqnarray}%
The normalising equation is%
\begin{equation*}
\sum\limits_{j=0}^{+\infty }p_{F,j}+\sum\limits_{j=0}^{+\infty
}\int\limits_{0}^{+\infty }p_{B,j}(x)dx+\sum\limits_{j=0}^{+\infty
}\int\limits_{0}^{+\infty }p_{D,j}(x)dx+\sum\limits_{j=0}^{+\infty
}\int\limits_{0}^{+\infty }\int\limits_{0}^{+\infty }p_{N,j}(x,y)dxdy
\end{equation*}%
\begin{equation*}
+\sum\limits_{j=0}^{+\infty }\int\limits_{0}^{+\infty }p_{I,j}(x)dx=1
\end{equation*}%
Define the generating functions%
\begin{eqnarray*}
P_{F}(z) &=&\dsum\limits_{j=0}^{\infty }p_{F,j}~z^{j}, \\
P_{B}(x,z) &=&\dsum\limits_{j=0}^{\infty }p_{B,j}(x)~z^{j}, \\
P_{D}(x,z) &=&\dsum\limits_{j=0}^{\infty }p_{D,j}(x)~z^{j}, \\
P_{N}(x,y,z) &=&\dsum\limits_{j=0}^{\infty }p_{N,j}(x,y)~z^{j}, \\
P_{I}(x,z) &=&\dsum\limits_{j=0}^{\infty }p_{I,j}(x)~z^{j},
\end{eqnarray*}%
for $\left\vert z\right\vert \leq 1.$\newline
We introduce $h(z)=\left[ \nu -\nu _{n}\beta _{n}\left( \lambda -\lambda
z\right) +\lambda -\lambda z\right] $ and\newline
$\chi \left( z\right) =h(z)-\nu _{d}\beta _{d}\left( \lambda -\lambda
z\right) \left( 1-\beta \left( h(z)\right) \right) $ to simplify notation.

We have the following theorem

\begin{theorem}
In steady state, the joint distribution of the server state and queue length
is given by%
\begin{eqnarray*}
P_{F}(z) &=&P_{F}(1)\exp \left\{ \int\limits_{1}^{z}\Psi (u)du\right\} , \\
P_{B}(x,z) &=&P_{B}(0,z)\left( 1-B(x)\right) \exp \left\{ -h(z)x\right\} , \\
P_{D}(x,z) &=&P_{B}(0,z)\frac{\nu _{d}\left( 1-\beta (h(z))\right) }{h(z)}%
\left( 1-B_{d}(x)\right) \exp \left\{ -\left( \lambda -\lambda z\right)
x\right\} , \\
P_{N}(x,y,z) &=&P_{B}(0,z)\nu _{n}\left( 1-B_{n}(y)\right) \left(
1-B(x)\right) \exp \left\{ -h(z)x\right\} \exp \left[ -\left( \lambda
-\lambda z\right) y\right] , \\
P_{I}(x,z) &=&\nu _{i}P_{F}(z)\left( 1-B_{i}(x)\right) \exp \left[ -\left(
\lambda -\lambda z\right) x\right] ,
\end{eqnarray*}%
where%
\begin{eqnarray*}
P_{F}(1) &=&\frac{\nu _{d}^{2}\beta \left( \nu _{d}\right) \left( 1-\rho
\right) }{\left( 1+\nu _{i}\beta _{1}^{i}\right) \left[ \nu _{d}^{2}\beta
\left( \nu _{d}\right) +\lambda \nu _{n}\beta _{1}^{n}\left( 1-\beta (\nu
_{d})\right) ^{2}-\lambda \nu _{n}\beta _{1}^{n}\nu _{d}\left( 1-\beta (\nu
_{d})\right) \right] }, \\
\Psi (z) &=&\frac{\lambda h(z)\beta \left( h(z)\right) -\left[ \lambda +\nu
_{i}\left( 1-\beta _{i}\left( \lambda -\lambda z\right) \right) \right] \chi
\left( z\right) }{\theta \left( z\chi \left( z\right) -h(z)\beta \left(
h(z)\right) \right) }, \\
P_{B}(0,z) &=&\frac{h(z)\left( \lambda +\theta \Psi (z)\right) }{\chi \left(
z\right) }P_{F}(z).
\end{eqnarray*}
\end{theorem}

\end{document}